\documentclass[12pt]{article}
\usepackage[T1]{fontenc}
\usepackage{multicol,epsfig,amssymb,amsmath,mathrsfs,verbatim} 
\usepackage{xcolor}
\newcommand{\proof}{\medskip \noindent {\bf Proof. }}
\newcommand{\qed}{\null\hfill $\Box\;\;$ \medskip}
\newcommand{\QED}{\null\hfill \Box\;\; \medskip}

\newtheorem {theorem} {Theorem}

\newtheorem {definition} {Definition}
\newtheorem {lemma} {Lemma}
\newtheorem {remark} {Remark}
\newtheorem {corrolary} {Corrolary}

\begin{document}

\title{A three dimensional modification of  the Gaussian number field}
\author{  J\'an Halu\v{s}ka, Ma\l{}gorzata Jastrz\k{e}bska}
\date{}
\maketitle

\def\R{\mathbb{R}}
\def\C{\mathbb{C}}
\def\N{\mathbb{N}}
\def\Z{\mathbb{Z}}
\def\Q{\mathbb{Q}}
\def\Ra{\operatorname{\textrm{rank}}}
\def\sign{\operatorname{sign}}
\def\ker{\operatorname{ker}}
\def\x{\times}
\def\<{\langle}
\def\>{\rangle}

\begin{abstract}
	 For vectors in $\mathbf{E}_3$ we introduce an  associative, commutative and  distributive  multiplication. We describe the related algebraic and geometrical properties, and hint some applications.
	
	Based on properties of hyperbolic (Clifford) complex numbers, we prove that the resulting algebra $\mathbb{T}$ is an associative algebra over a field and contains a subring isomorphic to hyperbolic complex numbers.  Moreover, the algebra  $\mathbb{T}$ is isomorphic to direct product $\mathbb{C}\times \mathbb{R}$, and so it contains a subalgebra isomorphic to the Gaussian complex plane.
	
\end{abstract}
	
	
	
\noindent {\it Acknowledgements.} The paper is supported by  VEGA Agency under grant 2/0106/19.
					
				The authors are grateful to the referee for careful reading of the paper and valuable suggestions and comments.
					
					\noindent{\it Author's address 1:}  Mathematical Institute, Slovak Academy of Sciences, Gre\v{s}\'{a}kova 6, Ko\v{s}ice, Slovakia,
jhaluska @ saske.sk

\noindent{\it Author's address 2:}
 Siedlce University of Natural Sciences and Humanities, Poland, majastrz2@wp.pl

\section{Introduction - how to model  colour vision?}
In a simplified explanation and for various purposes, three colours (red R, blue B, green G) satisfactorily model  human colour vision.

The  approximate utilization  of complex plane structure  is natural and commonly accepted  because  the eye retina  is flat. In detail, a real vector space  operations  in the plane are sufficient for modelling  the black-white vision. Operation of multiplication with linearly dependent R, G, B inputs could model the colour shade mixing. Such a planar approximation  of vision  is used in  construction of colour  TV-screens,   colour photography, colour painting, etc.
Note that  all these kinds of  illusions of reality in the human brain  usually use two  successive  reflections.
	
In \cite{Gregor-Haluska1}, a Gaussian complex plane spanned over three \textit{linearly  dependent} non  collinear non-zero vectors  is constructed. Having in mind the R, G, B colour decomposition  of the white light, a point in the Gaussian complex plane is a sum of three colours of various intensity (the so-called wheel of colours in optics).

Via  mathematics developed in this paper, we are able to  work  in the Euclidean three-dimensional geometrical space and  the biological (human) vision  is modelled  as a unique  reflection of light to a plane,  biologically it means one projection to the retina.
In  \cite{Haluska2},  the R, G, B triples of colours   are  represented via  functionals  modifying the approach  from \cite{Gregor-Haluska1}.

In the present paper, an  algebra $\mathbb{T}$ over field $\mathbb{R}$ is equipped with the basis (in $\mathbb{E}_3$)
  $$\{ {{\mathbf{1}_\mathbb{T}}} =(1,0,0), \mathbf{u}=(0,1,0), \mathbf{v}=(0,0,1)\}$$ and  the multiplication
	$$ \begin{array}{c|ccc}
	\otimes           & {{\mathbf{1}_\mathbb{T}}} & \mathbf{u} & \mathbf{v}   \\  \hline
	{{\mathbf{1}_\mathbb{T}}}           & {{\mathbf{1}_\mathbb{T}}} & \mathbf{u} & \mathbf{v}  \\
	\mathbf{u}           & \mathbf{u} & \mathbf{v} & -{{\mathbf{1}_\mathbb{T}}}  \\
	\mathbf{v}           & \mathbf{v} & -{{\mathbf{1}_\mathbb{T}}} & -\mathbf{u},
	\end{array}  $$
 cf. Definition~\ref{D1}.
It is isomorphic to the algebra $\mathbb{T}_1$ with the basis $ \{ k_1, k_2, k_3\}$, and the multiplication
	$$ \begin{array}{c|ccc}
	\cdot            & k_1 & k_2 & k_3   \\  \hline
	k_1       & k_1 & k_2 & 0  \\
	k_2    & k_2 &  -k_1 & 0  \\
	k_3   & 0 & 0 & k_3.
	\end{array}  $$

 Thus, the algebra $\mathbb{T}$  is well-known and was studied for example in
\cite{Alpay}, \cite{Lipatov},  \cite{DK94},  \cite{Snyder}.

 Some of the results of the present paper are known. 
 
 The goal of this paper  is mainly to show that
 \begin{itemize}
 \item   the algebra $\mathbb{T}$ (with its arithmetic, geometrical and topological structures) is  a spacial  phenomenon (like $\mathbb{C}$ is a phenomenon in the plane);
 \item an annihilator $\mathbb{D}$ of the subalgebra $\mathbb{G}$ in algebra $\mathbb{T}$ (defined in the paper) is only  a line in the space. Its all Lebesgue measurable subsets are of measure zero. Taking into  account specific properties of the space $\mathbb{E}_3$, a specific (=spacial) infinitesimal analysis  can be created for the Euclidean 3-dimensional space.
 \end{itemize}
For the sake of the paper, let us remind some known facts.

Each associative division algebra over the real number field of finite dimension $n \in \mathbb{N}$ is isomorphic (1) to  $\mathbb{R}$ (the field of all real numbers, $n=1$), or, (2) to $\mathbb{C}$ (the field of all Gaussian complex numbers, $n=2$), or, (3) to $\mathbb{H}$ (the algebra of all quaternions, $n=4$) by the  1877 theorem by G. F. Frobenius, cf. e.g.,~\cite{Nechaev}, p.~174.

Let   $\imath, \jmath, \varepsilon$ denote complex units  for three types of complex numbers, respectively. Remind that for  elliptic (Gaussian) numbers we have $\imath^2 = -1, \|\imath\|=1$; for hyperbolic (Clifford) numbers we have $\jmath^2=1,\|\jmath\|=1$; and, for parabolic (Studdy) numbers we have $\varepsilon^2=0,\|\varepsilon\|=1$.
Together all three types of  complex numbers are  called the \textit{generalized complex numbers}. Under a simplified term \textit{complex numbers} are usually understood the elliptic (Gaussian) complex numbers.  For details and geometrical aspects of the generalized complex numbers, the reader is referred, e.g., to~\cite{Harkin}.

A Hausdorff topology on $\mathbb{T}$ is given via an absolute  value, cf. Section~\ref{abs}.

\section{Operation of multiplication of vectors}
Concerning operation of addition, it is known that elements of the space $\mathbb{E}_3$ form an additive group with null $(0,0,0) \stackrel{def}{=} \Lambda$. Let 
$${{\mathbf{1}_\mathbb{T}}} \stackrel{def}{=}(1,0,0), \mathbf{u}\stackrel{def}{=}(0,1,0), \mathbf{v}\stackrel{def}{=}(0,0,1).$$
The set $\{ {{\mathbf{1}_\mathbb{T}}},  \mathbf{u}, \mathbf{v}\}$  is a  basis of the three dimensional  vector space $\mathbb{E}_3$ over real line $\mathbb{R}$.
So,  every element $\mathbf{x} \in \mathbb{E}_3$ can be written as
$$\mathbf{x} \stackrel{def}{=}
X_{{\mathbf{1}_\mathbb{T}}} {{\mathbf{1}_\mathbb{T}}} \oplus X_{\mathbf{u}} \mathbf{u} \oplus X_{\mathbf{v}} \mathbf{v}, $$
where $X_{{\mathbf{1}_\mathbb{T}}}, X_{\mathbf{u}},  X_{\mathbf{v}} \in \mathbb{R}$. The sign $\oplus$ denotes  an usual  parallelepiped   addition  in the vector space $\mathbb{E}_3$ and the sign $\ominus$ denotes  its inverse group operation;
we write also $a \ominus b \stackrel{def}{=}  a \oplus(\ominus b)$ so $a\ominus a = \Lambda$.

\begin{definition}\label{D1}\rm
	  Let $\mathbf{x}= (X_{{{\mathbf{1}_\mathbb{T}}}}, X_\mathbf{u}, X_\mathbf{v}) \in \mathbb{E}_3$ and $\mathbf{y}= (Y_{{\mathbf{1}_\mathbb{T}}}, Y_\mathbf{u}, Y_\mathbf{v}) \in \mathbb{E}_3$.  Then,	
\begin{multline*}\mathbf{x}  \otimes \mathbf{y}
  \stackrel{def}{=} {{\mathbf{1}_\mathbb{T}}}(X_{{\mathbf{1}_\mathbb{T}}} Y_{{{\mathbf{1}_\mathbb{T}}}} - X_{\mathbf{u}}Y_{\mathbf{v}} - X_{\mathbf{v}}Y_{\mathbf{u}}  ) \\
  \oplus \mathbf{u} (X_{{{\mathbf{1}_\mathbb{T}}}} Y_{\mathbf{u}} + X_{\mathbf{u}} Y_{{{\mathbf{1}_\mathbb{T}}}} - X_{\mathbf{v}}Y_{\mathbf{v}}) \\
 \oplus \mathbf{v} ( X_{{{\mathbf{1}_\mathbb{T}}}}Y_{\mathbf{v}}+X_{\mathbf{u}}
 Y_{\mathbf{u}} + X_{\mathbf{v}}Y_{{{\mathbf{1}_\mathbb{T}}}} ).
 \end{multline*}
 Let  $\mathbb{T}$  denote  the vector space  $\mathbb{E}_3$ over $\mathbb{R}$ equipped with this operation of multiplication.  \end{definition}	

 \begin{remark} \rm The operation of multiplication in $\mathbb{E}_3$ can be equivalently introduced via the  multiplication of basic elements as follows:
 \begin{multline}{{\mathbf{1}_\mathbb{T}}}\otimes {{\mathbf{1}_\mathbb{T}}} = {{\mathbf{1}_\mathbb{T}}}, \mathbf{u}\otimes \mathbf{u} = \mathbf{v}, \mathbf{v}\otimes \mathbf{v} =   \ominus \mathbf{u},\\
 {{\mathbf{1}_\mathbb{T}}} \otimes \mathbf{u} = \mathbf{u} \otimes {{\mathbf{1}_\mathbb{T}}} = \mathbf{u}, {{\mathbf{1}_\mathbb{T}}}\otimes \mathbf{v} =
 \mathbf{v}\otimes {{\mathbf{1}_\mathbb{T}}}= \mathbf{v},  \mathbf{u}\otimes \mathbf{v} = \mathbf{v}\otimes \mathbf{u}=\ominus{{\mathbf{1}_\mathbb{T}}}.\end{multline}

For example, geometrically, the entries of the table are vertexes  of the regular octahedron in~$\mathbb{E}_3$, the structure of the operation of multiplication becomes clearly visible,
$$ \begin{array}{c||ccc|ccc}
	\otimes       & {{\mathbf{1}_\mathbb{T}}} & \mathbf{u} & \mathbf{v}
	& -{{\mathbf{1}_\mathbb{T}}} &  -\mathbf{u} & - \mathbf{v} \\
	\hline \hline
	{{\mathbf{1}_\mathbb{T}}}    & {{\mathbf{1}_\mathbb{T}}} & \mathbf{u} & \mathbf{v}
	& -{{\mathbf{1}_\mathbb{T}}} &  -\mathbf{u} & - \mathbf{v} \\
	   \mathbf{u}    & \mathbf{u} & \mathbf{v} & -{{\mathbf{1}_\mathbb{T}}}
	& -\mathbf{u} &  -\mathbf{v} &  {{\mathbf{1}_\mathbb{T}}} \\
		\mathbf{v}    & \mathbf{v} & -{{\mathbf{1}_\mathbb{T}}} & -\mathbf{u}
	& -\mathbf{v} &  {{\mathbf{1}_\mathbb{T}}} &  \mathbf{u}  \\ \hline
		-{{\mathbf{1}_\mathbb{T}}}    & -{{\mathbf{1}_\mathbb{T}}} & -\mathbf{u} & -\mathbf{v}
		& {{\mathbf{1}_\mathbb{T}}} &  \mathbf{u} &  \mathbf{v} \\ 	
		-\mathbf{u}    & -\mathbf{u} & -\mathbf{v} & {{\mathbf{1}_\mathbb{T}}}
		& \mathbf{u} &  \mathbf{v} & - {{\mathbf{1}_\mathbb{T}}} \\
		-\mathbf{v}    & -\mathbf{v} & {{\mathbf{1}_\mathbb{T}}} & \mathbf{u}
		& \mathbf{v} &  -{{\mathbf{1}_\mathbb{T}}} & - \mathbf{u}.
	\end{array}  $$
\end{remark}

\begin{definition}\label{algebra over field}\rm

Let $\mathbb{K}$ be a field. An algebra over $\mathbb{K}$ is a vector space $A$ over $\mathbb{K}$ together with a bilinear associative multiplication (denoted by $\cdot$). \\
In other words, for arbitrary elements $a,b,c$ from a vector space $A$ and for arbitrary $\lambda$ from  $\mathbb{K}$, the following equalities are satisfied:\\
 1) $a\cdot (b+c)=a\cdot b+a\cdot c$;\\
 2) $(b+c)\cdot a=b\cdot a+c\cdot a$;\\
3) $(a\cdot b)\cdot c=a\cdot (b\cdot c)$;\\
4)$(\lambda a)\cdot b=a\cdot (\lambda b)=\lambda (a\cdot b).$

Algebras over a field which also satisfy commutativity for multiplication are called commutative algebras over a field.
An algebra $A$ is said to be finite dimensional or infinite dimensional according to whether the space $A$ is finite dimensional or infinite dimensional.
An algebra $A$ is unital if it has an identity (unit) element with respect to the multiplication. An ideal of unital algebra $A$ is a linear subspace which is also an ideal in $A$ as a ring.

\end{definition}	

It follows from the bilinearity of the multiplication in algebra over a field that, given a basis $\{a_1, a_2, \ldots , a_n\}$ of the space $A$, the multiplication is uniquely determined by
the products of the basic vectors $a_i\cdot a_j.$ It is sufficient  to prove associativity of multiplication only for basic vectors.
For more details on associative algebras over a field, we refer the reader to \cite{DK94}, \cite{Pierce}.

\begin{theorem}
The algebra $\mathbb{T}$ is three-dimensional unital associative and commutative algebra over field $\mathbb{R}.$

\end{theorem}

\begin{proof}
The commutativity of multiplication and  the distributivity  can be easily checked from the definition of algebra $\mathbb{T}.$ It is easy to see that identity element in $\mathbb{T}$ is equal to ${{\mathbf{1}_\mathbb{T}}}=(1,0,0).$
Proof of the associativity
$$ [\mathbf{x}\otimes\mathbf{y}] \otimes\mathbf{z} = \mathbf{x}\otimes[\mathbf{y} \otimes\mathbf{z}]$$
needs a rather longer  but only technical calculations.
\end{proof}

\section{$\sigma$-Conjugation,   a homomorphism  of $\mathbb{T}$  to $\sigma$}
 \begin{definition}\rm
If
$$\mathbf{x}= {{\mathbf{1}_\mathbb{T}}} X_{{\mathbf{1}_\mathbb{T}}} \oplus \mathbf{u} X_\mathbf{u} \oplus \mathbf{v} X_\mathbf{v}   \in  \mathbb{T},$$
then  we define a \textit{$\sigma$-conjugate element  $\mathbf{x}^*$ of the element} $\mathbf{x}$ as follows:
$$\mathbf{x}^* \stackrel{def}{=}{{\mathbf{1}_\mathbb{T}}} X_{{\mathbf{1}_\mathbb{T}}} \ominus \mathbf{u} X_\mathbf{v} \ominus \mathbf{v} X_\mathbf{u}  \in  \mathbb{T},$$
where $ X_{{\mathbf{1}_\mathbb{T}}}, X_\mathbf{u}, X_\mathbf{v} \in   \mathbb{R}.$
\end{definition}

Note that the   vectors  $\delta\stackrel{def}{=} \mathbf{u} \ominus \mathbf{v}$ and  ${{\mathbf{1}_\mathbb{T}}}$ are perpendicular: the scalar product  $$<\mathbf{u}\ominus\mathbf{v}, {{\mathbf{1}_\mathbb{T}}}> = <(0,1,-1), (1,0,0)> = 0.$$

\begin{lemma} \rm
If	$$\mathbf{x}= {{\mathbf{1}_\mathbb{T}}} X_{{\mathbf{1}_\mathbb{T}}} \oplus \mathbf{u} X_\mathbf{u} \oplus \mathbf{v} X_\mathbf{v}  \in \mathbb{T},$$
 then

 \begin{equation}
  \mathbf{x} \otimes \mathbf{x}^* =   A(\mathbf{x}){{\mathbf{1}_\mathbb{T}}} \oplus B(\mathbf{x})\delta,  \label{eq}  \end{equation}
where
$A(\mathbf{x}) = X_{{\mathbf{1}_\mathbb{T}}}^2 + X_\mathbf{u}^2 + X_\mathbf{v}^2,  B(\mathbf{x}) = X_{{\mathbf{1}_\mathbb{T}}}X_\mathbf{u} + X_\mathbf{u}X_\mathbf{v} -   X_\mathbf{v}X_{{\mathbf{1}_\mathbb{T}}}, $

$\delta = \mathbf{u} \ominus \mathbf{v}$
and $ X_{{\mathbf{1}_\mathbb{T}}}, X_\mathbf{u}, X_\mathbf{v} \in   \mathbb{R}$. 	
\end{lemma}
\proof
By Definition~\ref{D1}, $$\mathbf{x} \otimes \mathbf{x}^* =  [{{\mathbf{1}_\mathbb{T}}} X_{{\mathbf{1}_\mathbb{T}}} \oplus \mathbf{u} X_\mathbf{u} \oplus \mathbf{v} X_\mathbf{v}] \otimes [{{\mathbf{1}_\mathbb{T}}} X_{{\mathbf{1}_\mathbb{T}}} \ominus \mathbf{u} X_\mathbf{v} \ominus \mathbf{v} X_\mathbf{u}]$$
$$= {{\mathbf{1}_\mathbb{T}}} (X_{{\mathbf{1}_\mathbb{T}}}^2 + X_\mathbf{u}^2 + X_\mathbf{v}^2)
\oplus \delta ( X_{{\mathbf{1}_\mathbb{T}}}X_\mathbf{u} + X_\mathbf{u}X_\mathbf{v} - X_\mathbf{v}X_{{\mathbf{1}_\mathbb{T}}} ). \hskip 1cm \QED$$

Recall that a subalgebra of an algebra $A$ over  field $\mathbb{K}$ is a subset of elements  that is closed under addition, multiplication, and scalar multiplication.

 \begin{theorem} \label{T4}
Let  $\sigma$ denote the linear subspace of $\mathbb{E}_3$ spanned by vectors $ {{\mathbf{1}_\mathbb{T}}}$ and $\delta .$  Let $\boxtimes$ be the restriction of $\otimes$ on the subspace $\sigma$ and let
$$\mathbf{j}_\mathbb{T}=\frac{{{\mathbf{1}_\mathbb{T}}}}{3} \oplus \frac{2\delta}{3}= \frac{1}{3}\left[{{\mathbf{1}_\mathbb{T}}} \oplus 2\mathbf{u} \ominus 2\mathbf{v}\right].$$
Then, $\sigma$ is a subalgebra of algebra $\mathbb{T}$ and the operation $\boxtimes$ is a hyperbolic  complex multiplication on a plane $\sigma$. The ,,real"  unit is ${{\mathbf{1}_\mathbb{T}}}$ and the ,,imaginary" unit is $\mathbf{j}_\mathbb{T}$, respectively, i.e.,
$$ \mathbf{j}_\mathbb{T} \boxtimes \mathbf{j}_\mathbb{T}= {{\mathbf{1}_\mathbb{T}}}.$$
 \end{theorem}

 \proof
In order to prove that $\sigma$ is a subalgebra of $\mathbb{T}$, it suffices  to show that for all $\mathbf{x},\mathbf{y}\in\sigma$ their product $\mathbf{x}\boxtimes \mathbf{y}= \mathbf{x} \oplus \mathbf{y}$ belongs to $\sigma$. This follows easily  from direct calculations.
A question of the  length of an  unit ,,imaginary" element $\mathbf{j}_\mathbb{T}$ will be  solved after introducing  a notion of the absolute  value  $\|\cdot \|$ on $\mathbb{T}$ such that  $\|\mathbf{j}_\mathbb{T}\|=1$, cf.  Theorem~\ref{T7}, (vii).

 To prove that $\sigma$ is a hyperbolic complex plane,
 it is sufficient  to show that
$$ \mathbf{j}_\mathbb{T}\boxtimes \mathbf{j}_\mathbb{T}= {{\mathbf{1}_\mathbb{T}}},$$ and  this follows from direct calculations.
\qed

The following lemma is useful.
\begin{lemma} \label{MJ3}

If $\mathbf{x}, 
\mathbf{y}
\in \mathbb{T}$, then
	$ (\mathbf{x}\otimes\mathbf{y})^*= \mathbf{x}^*\otimes \mathbf{y}^*.$
\end{lemma}
\proof
Let us denote $\mathbf{x}= (X_{{\mathbf{1}_\mathbb{T}}},X_\mathbf{u},X_\mathbf{v}), \mathbf{y}= (Y_{{\mathbf{1}_\mathbb{T}}},Y_\mathbf{u},Y_\mathbf{v})$ and
$\mathbf{x}^* =(X_{{\mathbf{1}_\mathbb{T}}}',X_\mathbf{u}',X_\mathbf{v}')=(X_{{\mathbf{1}_\mathbb{T}}},-X_\mathbf{v},-X_\mathbf{u})$ and $\mathbf{y}^*= (Y_{{\mathbf{1}_\mathbb{T}}}',Y_\mathbf{u}',Y_\mathbf{v}') =
(Y_{{\mathbf{1}_\mathbb{T}}},- Y_\mathbf{v}, -Y_\mathbf{u})$, respectively. We have:
$$ \mathbf{x}^* \otimes \mathbf{y}^* = $$
$${{\mathbf{1}_\mathbb{T}}}(X_{{\mathbf{1}_\mathbb{T}}}'Y_{{\mathbf{1}_\mathbb{T}}}'-X_\mathbf{u}'Y_\mathbf{v}'- X_\mathbf{v}'Y_\mathbf{u}')\oplus \mathbf{u} (X_{{\mathbf{1}_\mathbb{T}}}' Y_\mathbf{u}' + X_\mathbf{u}'Y_{{\mathbf{1}_\mathbb{T}}}'- X_\mathbf{v}'Y_\mathbf{v}') \oplus \mathbf{v} (X_{{\mathbf{1}_\mathbb{T}}}'Y_\mathbf{v}'+X_\mathbf{u}'Y_\mathbf{u}'+ X_\mathbf{v}'Y_{{\mathbf{1}_\mathbb{T}}}')$$
\begin{multline*}= {{\mathbf{1}_\mathbb{T}}}(X_{{\mathbf{1}_\mathbb{T}}}Y_{{\mathbf{1}_\mathbb{T}}}-X_\mathbf{v}Y_\mathbf{u}- X_\mathbf{u}Y_\mathbf{v}) \\ \oplus \mathbf{u} (-X_{{\mathbf{1}_\mathbb{T}}} Y_\mathbf{v} -X_\mathbf{v}Y_{{\mathbf{1}_\mathbb{T}}}- X_\mathbf{u}Y_\mathbf{u}) \\ \oplus \mathbf{v} (-X_{{\mathbf{1}_\mathbb{T}}}Y_\mathbf{u}+X_\mathbf{v}Y_\mathbf{v}-X_\mathbf{u}Y_{{\mathbf{1}_\mathbb{T}}})
\end{multline*} $$ =(\mathbf{x}\otimes\mathbf{y})^*. \hskip 1cm \QED $$

\section{A shift of the plane $\sigma$}
Let us make the following regular linear transformation  of the plane $\sigma$: \\
 $$( {{\mathbf{1}_\mathbb{T}}}, \delta)\to ( {{\mathbf{1}_\mathbb{T}}}, \mathbf{j}_\mathbb{T}),$$
where
 $$ \mathbf{j}_\mathbb{T}=\frac{{{\mathbf{1}_\mathbb{T}}}}{3} \oplus \frac{2\delta}{3}, \delta = \frac{3 \mathbf{j}_\mathbb{T}}{2} \ominus \frac{{{\mathbf{1}_\mathbb{T}}}}{2}.$$
Then, $$ \mathbf{x} \otimes \mathbf{x}^*
 =\mathcal{A}(\mathbf{x}) {{\mathbf{1}_\mathbb{T}}}\oplus \mathcal {B}(\mathbf{x}) \mathbf{j}_\mathbb{T},$$
where
$$\mathcal{A}(\mathbf{x}) = (X_{{\mathbf{1}_\mathbb{T}}}^2 + X_\mathbf{u}^2 + X_\mathbf{v}^2) - \frac{X_{{\mathbf{1}_\mathbb{T}}} X_\mathbf{u} + X_\mathbf{u} X_\mathbf{v} - X_{{\mathbf{1}_\mathbb{T}}} X_\mathbf{v}}{2} $$
and
$$\mathcal{B}(\mathbf{x}) = \frac{3(X_{{\mathbf{1}_\mathbb{T}}} X_\mathbf{u} + X_\mathbf{u} X_\mathbf{v} - X_{{\mathbf{1}_\mathbb{T}}} X_\mathbf{v})}{2}.$$

\begin{remark} \rm If it does not lead to ambiguities, we usually omit arguments
in functions $A(\mathbf{x}), B(\mathbf{x}), \mathcal{A}(\mathbf{x}), \mathcal{B}(\mathbf{x})$ and  white simply  $A, B, \mathcal{A}, \mathcal{B}$, respectively.
\end {remark}
  Since  ${{\mathbf{1}_\mathbb{T}}} \otimes {{\mathbf{1}_\mathbb{T}}} = {{\mathbf{1}_\mathbb{T}}}$ and $\mathbf{j}_\mathbb{T} \otimes \mathbf{j}_\mathbb{T} = {{\mathbf{1}_\mathbb{T}}}$, by Theorem~\ref{T4},
     $$ (\mathcal{A}{{\mathbf{1}_\mathbb{T}}} \ominus \mathcal{B} \mathbf{j}_\mathbb{T}) \otimes (\mathcal{A}{{\mathbf{1}_\mathbb{T}}} \oplus \mathcal{B}\mathbf{j}_\mathbb{T}) =
       [(\mathcal{A}^2  {{\mathbf{1}_\mathbb{T}}} \otimes {{\mathbf{1}_\mathbb{T}}}) \ominus (\mathcal{B}^2  \mathbf{j}_\mathbb{T}\otimes \mathbf{j}_\mathbb{T})] $$
  $$ = \mathcal{A}^2 {{\mathbf{1}_\mathbb{T}}} \ominus \mathcal{B}^2{{\mathbf{1}_\mathbb{T}}} =  (\mathcal{A}^2 - \mathcal{B}^2){{\mathbf{1}_\mathbb{T}}}  =
 (\mathcal{A + B})\cdot (\mathcal{A - B}) {{\mathbf{1}_\mathbb{T}}}.$$

  \section{Expression $ \mathcal{A + B}$}
\begin{lemma}\label{L8}
	Let $\mathbf{x} = (X_{{\mathbf{1}_\mathbb{T}}}, X_\mathbf{u}, X_\mathbf{v}) \in \mathbb{T}$. Then,
	$$\mathcal{A} + \mathcal{B} = A+B
	= \frac{(X_{{\mathbf{1}_\mathbb{T}}} + X_\mathbf{u})^2}{2} + \frac{(X_\mathbf{u} + X_\mathbf{v})^2}{2} +\frac{(X_{{\mathbf{1}_\mathbb{T}}} - X_\mathbf{v})^2}{2},$$
	where $$A = X_{{\mathbf{1}_\mathbb{T}}}^2 + X_\mathbf{u}^2 + X_\mathbf{v}^2,
	B = X_{{\mathbf{1}_\mathbb{T}}} X_\mathbf{u} + X_\mathbf{u} X_\mathbf{v} - X_{{\mathbf{1}_\mathbb{T}}} X_\mathbf{v}.$$
\end{lemma}
\proof It is easy to verify that
\begin{multline}\mathcal{A} + \mathcal{B}= A+B=(X_{{\mathbf{1}_\mathbb{T}}}^2+ X_\mathbf{u}^2 + X_\mathbf{v}^2) + (X_\mathbf{u}X_\mathbf{v} + X_\mathbf{u}X_{{\mathbf{1}_\mathbb{T}}} -X_{{\mathbf{1}_\mathbb{T}}}X_\mathbf{v}) \\
	= \frac{(X_{{\mathbf{1}_\mathbb{T}}} + X_\mathbf{u})^2}{2} + \frac{(X_\mathbf{u} + X_\mathbf{v})^2}{2} +\frac{(X_{{\mathbf{1}_\mathbb{T}}} - X_\mathbf{v})^2}{2}. \label{star}  \QED
\end{multline}

	 \begin{definition} \rm    Let  $\Gamma= (X_{{\mathbf{1}_\mathbb{T}}}, X_\mathbf{u},X_\mathbf{v})  \in \mathbb{T}$. Denote  $$\mathbb{D} \stackrel{def}{=} \{\Gamma  \in \mathbb{T}\ | \ A(\Gamma)+B(\Gamma)\}=0\},$$
     where $$A(\Gamma) = X_{{\mathbf{1}_\mathbb{T}}}^2 + X_\mathbf{u}^2 +X_\mathbf{v}^2$$
and
$$B(\Gamma) =  X_{{\mathbf{1}_\mathbb{T}}}  X_\mathbf{u} +X_\mathbf{u} X_\mathbf{v} - X_\mathbf{v}X_{{\mathbf{1}_\mathbb{T}}}.$$      	\end{definition}

\begin{lemma}\label{MJ}
$ \mathbb{D} = \{\Gamma \in \mathbb{T} \mid
            \exists  \gamma \in \mathbb{R},
           \Gamma =( \gamma, -\gamma, \gamma)\}.$
\end{lemma}
 \proof
            If $\Gamma =(\gamma, -\gamma, \gamma)$,  $\gamma \in \mathbb{R}$,  then
              $A(\Gamma) =  -B(\Gamma) =3 \gamma^2$. So, $A(\Gamma) +B(\Gamma) = 0$.
             Now,  let $\Gamma \in D$. Then  $A(\Gamma) + B(\Gamma) = 0$, and from Lemma \ref{L8}, it follows that             
   the system of equations
     $$\left\{ \begin{array}{c} X_{{\mathbf{1}_\mathbb{T}}} + X_\mathbf{u}= 0 \\ X_\mathbf{u} + X_\mathbf{v} = 0 \\ X_{{\mathbf{1}_\mathbb{T}}} -X_\mathbf{v} = 0 \end{array}\right.$$
     is satisfied.
     Hence,
        $X_{{\mathbf{1}_\mathbb{T}}} =  \gamma, X_\mathbf{u} = - \gamma, X_\mathbf{v} = \gamma $, for some    $\gamma \in \mathbb{R}. \QED     $

          \begin{remark}
          Vectors $(\gamma, -\gamma, \gamma)$, $\gamma \in \mathbb{R}$,  form a  line, it is a diagonal  of the fourths and sixths octants of the space~$\mathbb{E}_3$.
          \end{remark}

         \begin{lemma} \label{MJ2}     Let
	$\mathbf{x} = (X_{{\mathbf{1}_\mathbb{T}}},X_\mathbf{u},X_\mathbf{v}) \in \mathbb{T}$.
	Then, $$\mathbf{x}\otimes \mathbf{x}^* = (A, B, -B) =
	(A + B){{\mathbf{1}_\mathbb{T}}} \oplus \Gamma, $$ where
$$A=X_{{\mathbf{1}_\mathbb{T}}}^2+ X_\mathbf{u}^2 + X_\mathbf{v}^2, B =  X_{{\mathbf{1}_\mathbb{T}}} X_\mathbf{u} +  X_\mathbf{v} X_\mathbf{u} - X_{{\mathbf{1}_\mathbb{T}}}X_\mathbf{v},$$ and $\Gamma =(-B,B,-B) \in \mathbb{D}.$  \end{lemma}

\proof A  verification is  trivial. 
\qed

\section{Expression $ \mathcal{A - B}$}  \label{AH}
\begin{lemma}\label{L9}
Let $\mathbf{x} = (X_{{\mathbf{1}_\mathbb{T}}}, X_\mathbf{u}, X_\mathbf{v}) \in \mathbb{T}$. Then,
$$ A -B =  \mathcal{A - B} = \frac{(X_{{\mathbf{1}_\mathbb{T}}} - X_\mathbf{u})^2}{2}+\frac{(X_\mathbf{u} - X_\mathbf{v})^2}{2} +\frac{(X_{{\mathbf{1}_\mathbb{T}}} + X_\mathbf{v})^2}{2}.$$
\end{lemma}
\proof
We have
            $$ \mathcal{A - B} = A -B =(X_{{\mathbf{1}_\mathbb{T}}}^2+ X_\mathbf{u}^2 + X_\mathbf{v}^2) + (X_{{\mathbf{1}_\mathbb{T}}}X_\mathbf{v} -  X_\mathbf{u}X_\mathbf{v} - X_\mathbf{u}X_{{\mathbf{1}_\mathbb{T}}}) $$
            $$= \frac{(X_{{\mathbf{1}_\mathbb{T}}} - X_\mathbf{u})^2}{2}+\frac{(X_\mathbf{u} - X_\mathbf{v})^2}{2} +\frac{(X_{{\mathbf{1}_\mathbb{T}}} + X_\mathbf{v})^2}{2}. \hskip20mm \QED $$

\section{Algebraic structure of $\mathbb{T}$}

\begin{lemma} \label{Lemma10} The set $\mathbb{D}$   is an ideal in the algebra $\mathbb{T}$.
\end{lemma}
 \proof From Lemma \ref{MJ}, $\mathbb{D}$ is  the one-dimensional linear space spanned by vector $(1,-1,1).$
Let    $\Gamma=\gamma(1,-1,1) \in \mathbb{D}$, where $\gamma \in \mathbb{R}$. Let $\mathbf{x}=(X_{{\mathbf{1}_\mathbb{T}}},X_\mathbf{u}, X_\mathbf{v}) \in \mathbb{T}$, where
 $X_{\mathbf{1}_\mathbb{T}},X_\mathbf{u},X_\mathbf{v} \in \mathbb{R}$.
 Then, $\Gamma \otimes \mathbf{x} \in \mathbb{D}$.   Indeed,
$$\begin{array}{rcl} \Gamma \otimes \mathbf{x} &=& (\gamma, -\gamma,  \gamma) \otimes (X_{{\mathbf{1}_\mathbb{T}}},X_\mathbf{u}, X_\mathbf{v})\\
&=& (X_{{\mathbf{1}_\mathbb{T}}}\gamma -X_\mathbf{u}\gamma + X_\mathbf{v}\gamma,  -X_{{\mathbf{1}_\mathbb{T}}}\gamma + X_\mathbf{u}\gamma - X_\mathbf{v}\gamma, X_{{\mathbf{1}_\mathbb{T}}}\gamma -X_\mathbf{u}\gamma +X_\mathbf{v}\gamma)\\
&=&(\beta, -\beta, \beta) \in \mathbb{D},\end{array}$$
where $\beta=X_{{\mathbf{1}_\mathbb{T}}}\gamma -X_\mathbf{u}\gamma + X_\mathbf{v}\gamma  \in \mathbb{R}.$
Consequently,   $\mathbb{D}$   is an ideal in  $\mathbb{T}$.    \qed

\begin{lemma} \label{Lemma MJ2} The set $\mathbb{G} =  (\alpha - \beta, \alpha, \beta) \in \mathbb{T} \mid \alpha \in \mathbb{R}, \beta \in \mathbb{R}\}$  is an ideal in  $\mathbb{T}$.
\end{lemma}

\proof
It is easily seen that the set $\mathbb{G}$ is two-dimensional linear space spanned by vectors $(1,1,0)$ and $(-1,0,1).$
Let    $\theta=(\alpha -\beta, \alpha,\beta ) \in \mathbb{G}$ where $\alpha, \gamma \in \mathbb{R}$. Let $\mathbf{x}=(X_{{\mathbf{1}_\mathbb{T}}},X_\mathbf{u}, X_\mathbf{v}) \in \mathbb{T}$.
Then, $\theta \otimes \mathbf{x} \in \mathbb{G}$.  

Indeed,
$ \theta \otimes \mathbf{x} = (\alpha -\beta, \alpha,\beta ) \otimes (X_{{\mathbf{1}_\mathbb{T}}},X_\mathbf{u}, X_\mathbf{v})=
 ((\alpha - \beta)X_{{\mathbf{1}_\mathbb{T}}} -\alpha X_\mathbf{v}- \beta X_\mathbf{u},  (\alpha - \beta)X_\mathbf{u} + \alpha X_{{\mathbf{1}_\mathbb{T}}} - \beta X_\mathbf{v}, (\alpha - \beta)X_\mathbf{v}+\alpha X_\mathbf{u}\ +\beta X_\mathbf{v})
 \in \mathbb{G}.$

Consequently,    $\mathbb{G}$   is an ideal in  $\mathbb{T}.$    \qed

\begin{theorem} \label{MJ5} Let  $\mathbb{T}$   be an algebra described in Definition \ref{D1}.  Then, $\mathbb{T}$ is the direct sum of ideals $\mathbb{D}$ and $\mathbb{G}.$ Moreover, algebra $\mathbb{T}$ is isomorphic to $\mathbb{C}\times\mathbb{R}.$
\end{theorem}

\proof From definitions of  ideals $\mathbb{D}$ and $\mathbb{G}$ it follows that  $ \mathbb{D} \cap \mathbb{G} = \Lambda.$
The ideal $\mathbb{D}$ is a division ring, so, by Frobenius theorem, it is isomorphic to $\mathbb{R}$. The identity element $\mathbf{1}_\mathbb{D}$ of $\mathbb{D}$  equals $\mathbf{1}_\mathbb{D}= (\frac{1}{3}, -\frac{1}{3}, \frac{1}{3})$.
It is easy to verify that the identity element $\mathbf{1}_\mathbb{G}$ of $\mathbb{G}$ is equal to  $\mathbf{1}_\mathbb{G}=\left(\frac{2}{3},\frac{1}{3}, -\frac{1}{3}\right).$
 Moreover, it is trivial to calculate that the ideal $\mathbb{G}$  has element $\mathbf{i}_\mathbb{G}=  \left(0, \sqrt{\frac{1}{3}}, \sqrt{\frac{1}{3}}\right)$ which satisfies $\mathbf{i}_\mathbb{G} \otimes \mathbf{i}_\mathbb{G}= -\mathbf{1}_\mathbb{G}.$ This forces $\mathbb{G}$ to be  isomorphic to complex numbers  $\mathbb{C}.$
 \qed

Let  $\mathbb{A}$ be an algebra over a field $\mathbb{K}$ and $S \subset \mathbb{A}$. Recall  that  an \textit{annihilator of} $S$ in $\mathbb{A}$, denoted $\textrm {Ann}_\mathbb{A}(S)$, is the set of all elements $a \in \mathbb{A}$ such that $a \cdot s = 0$  for all $s \in S$.

In the following, we list several simple facts resulting directly from the Theorem \ref{MJ5}

\begin{corrolary}\

\begin{itemize}
\item $\textrm {Ann}_{\mathbb{T}}(\mathbb{D})=\mathbb{G}$ and
 $\textrm {Ann}_{\mathbb{T}}(\mathbb{G})=\mathbb{D}.$
\item  $\mathbf{y} \in \mathbb{T}$ is a zero divisor in $\mathbb{T}$ if and only if    $\mathbf{y}  \in \mathbb{D}$ or  $ \mathbf{y}  \in \mathbb{G}$.
\item The element $\mathbf{x} \in \mathbb{T}$ is invertible in the algebra $\mathbb{T}$   if and only if there exist $a,b,c \in \mathbb{R}$ such that
$$\mathbf{x} =   a(1,1,0) \oplus b(-1,0,1) \oplus c(1,-1,1), $$ where $a^2 + b^2 >0$ and $ c^2 >0$.
\end{itemize}

\end{corrolary}


\section{Absolute value  on $\mathbb{T}$ }\label{abs}
A Hausdorff topology on $\mathbb{T}$ is determined via an  absolute value, cf.~\cite{Rudin}.

     \begin{definition}\label{D5}\rm
     	Let $\mathbf{x} = (X_{{\mathbf{1}_\mathbb{T}}},X_\mathbf{u},X_\mathbf{v}) \in \mathbb{T}$. An \textit{absolute  value } of the element $\mathbf{x}$ is  a non-negative real number  $\|\mathbf{x}\|: \mathbf{x} \to [0,+ \infty)$ such that
     	$$ \|\mathbf{x}\|  =\sqrt{A + B},$$
     	where
     	$$A = X_{{\mathbf{1}_\mathbb{T}}}^2 + X_\mathbf{u}^2 + X_\mathbf{v}^2, \hskip 1cm B = X_{{\mathbf{1}_\mathbb{T}}}X_\mathbf{u} + X_\mathbf{u}X_\mathbf{v} -   X_\mathbf{v}X_{{\mathbf{1}_\mathbb{T}}}.$$
     \end{definition}
     The absolute  value  has the following  properties:

   \begin{theorem}\label{T7}	Let $\mathbf{x} = (X_{{\mathbf{1}_\mathbb{T}}},X_\mathbf{u},X_\mathbf{v}), \mathbf{y} = (Y_{{\mathbf{1}_\mathbb{T}}},Y_\mathbf{u},Y_\mathbf{v})  \in \mathbb{T}$;  $\gamma \in \mathbb{R}$, $\Gamma\in \mathbb{D}$.
   	
   	Then
     	\begin{itemize}
 \item[{(i)}] $\|\Gamma\| = 0 $;
 \item[{(ii)}] $\|\gamma \mathbf{x}\| = |\gamma| \cdot \| \mathbf{x}\|$;
 \item[(iii)] $\|\mathbf{x}\|\geq 0$;
\item[(iv)] $ \|\mathbf{x} \otimes \mathbf{y}\|=\|\mathbf{x}\| \cdot \|\mathbf{y}\|$;     		
\item[{(v)}] $\|\Gamma \oplus \mathbf{x}\| = \| \mathbf{x}\|$;
\item[(vi)]  $\|{{\mathbf{1}_\mathbb{T}}}\| = \|\mathbf{u}\|  = \|\mathbf{v}\| = 1$;
\item [(vii)] $\|\mathbf{j}_\mathbb{T}\| = 1$;  		
\item[{(viii)}]
$\|(0, - \gamma, \gamma)\| = \| (  \gamma, -\gamma, 0)\| = \|( \gamma, 0,\gamma)\|= \|( \gamma, 0,0)\|= \|( 0, -\gamma,  0)\| = \|(0,0,\gamma)\|= |\gamma| \geq 0$;
\item[{(ix)}] $\|\mathbf{x}\| {=} \sqrt{\| \mathbf{x}\otimes\mathbf{x}^*\|}$;
\item[{(x)}]
$ \|\mathbf{x} +\mathbf{y} \| \leq \|\mathbf{x}\| + \| \mathbf{y}\|$,
\item[{(xi)}]
$\left\|\left(\frac{2}{3},\frac{1}{3}, -\frac{1}{3}\right)\right\| = 1, \left\|\left(0, \sqrt{\frac{1}{3}} \sqrt{\frac{1}{3}}\right)\right\| = 1.$
\end{itemize}
 \end{theorem}
  \proof 	\begin{enumerate}
 	\item[(i)] There exists an $\gamma \in \mathbb{R}$  such that $\Gamma=(\gamma, -\gamma, \gamma)$ is a zero divisor. Then $\|\Gamma\| = 0$ from Definition~\ref{D5}.
 	\item[(ii)] This statement trivially follows from Definition~\ref{D5} and (\ref{star}).
 	\item[(iii)] It is a simple consequence of Lemma \ref{L8}.
 	
 	\item[(iv)]

	By Lemma \ref{MJ3}
	\begin{equation}(\mathbf{x} \otimes \mathbf{y})\otimes (\mathbf{x \otimes y})^*	=(\mathbf{x} \otimes \mathbf{y})\otimes (\mathbf{x}^* \otimes \mathbf{y}^*).\label{MJ4}  \end{equation}
	Applying	 Lemma \ref{MJ2} together with properties of multiplication and addition in $\mathbb{T}$,  we may prove  the following points\\	
	
		a) The right side  of Formula (\ref{MJ4}) can be written
	as\\
	$(\mathbf{x} \otimes \mathbf{y})\otimes (\mathbf{x \otimes y})^*	=$\\
	$( A(\mathbf{x\otimes y})   + B(\mathbf{x\otimes y})){{\mathbf{1}_\mathbb{T}}} \oplus (-B(\mathbf{x\otimes y}),B(\mathbf{x \otimes y}),-B(\mathbf{x \otimes y}))=$\\
	 $( A(\mathbf{x\otimes y})   + B(\mathbf{x\otimes y})){{\mathbf{1}_\mathbb{T}}} \oplus   \Gamma_2 $\\	where $\Gamma_2 = (-B(\mathbf{x\otimes y}),B(\mathbf{x \otimes y}),-B(\mathbf{x \otimes y})) \in \mathbb{D}.$	\\

	b) The left side  of  Formula (\ref{MJ4}) can be written as
	
	$(\mathbf{x} \otimes \mathbf{y})\otimes (\mathbf{x}^* \otimes \mathbf{y}^*)   =( \mathbf{x} \otimes \mathbf{x}^*) \otimes (\mathbf{y} \otimes \mathbf{y}^*) = $\\
	$\Big( \big(A(\mathbf{x})   + B(\mathbf{x})\big){{\mathbf{1}_\mathbb{T}}} \oplus \big(-B(\mathbf{x}),B(\mathbf{x}),-B(\mathbf{x})\big) \Big) \otimes 	\Big( \big(A(\mathbf{y}) + B(\mathbf{y})\big){{\mathbf{1}_\mathbb{T}}} \oplus \big(-B(\mathbf{y}),B(\mathbf{y}),-B(\mathbf{x}) \big)\Big) 	$.	
	Now, by Lemma \ref{Lemma10}, this is equivalent with \\ $\Big(A(\mathbf{x})   + B(\mathbf{x})\Big) \Big(A(\mathbf{y})   + B(\mathbf{y})\Big) {{\mathbf{1}_\mathbb{T}}}  + \Gamma_1$ for some $\Gamma_1 \in \mathbb{D}.$\\

	c) In fact, vectors contained in $\mathbb{D}$ are linearly independent from vectors $\alpha {{\mathbf{1}_\mathbb{T}}}$ for any $\alpha \in \mathbb{R}.$	
	Now, comparing the final formulas in items a) and b) we get that	\\
	$  A(\mathbf{x\otimes y})   + B(\mathbf{x\otimes y})= \Big(A(\mathbf{x})   + B(\mathbf{x})\Big) \Big(A(\mathbf{y})   + B(\mathbf{y})\Big) $.

 	Hence
 	$ \|\mathbf{x} \otimes \mathbf{y}\|=\|\mathbf{x}\| \cdot \|\mathbf{y}\|$.	
		
 	\item [(v)]
	
	Let  $\Gamma = (\gamma, -\gamma, \gamma).$  Then,	$\|\mathbf{x} \oplus \Gamma\| = \|\mathbf{x}\|$ is a simple consequence of   Lemma \ref{L8}.

 	\item[(vi)]  $\|{{\mathbf{1}_\mathbb{T}}}\| = \|\mathbf{u}\|  = \|\mathbf{v}\| = 1$ can be  calculated  directly from definition;

 	\item [(vii)]
 	$$\mathbf{j}_\mathbb{T} = \frac{1}{3}[{{\mathbf{1}_\mathbb{T}}}\oplus 2\mathbf{u}\oplus 2(-\mathbf{v})]$$ $$= \frac{1}{3}[(1,0,0)\oplus 2(0,1,0)\oplus 2(0,0,-1)]=
 	\left(\frac{1}{3},\frac{2}{3},-\frac{2}{3}\right),$$
 	and
 	$$\|\mathbf{j}_\mathbb{T}\|= \sqrt{A+B} $$
 	$$ =\sqrt{\left[\left(\frac{1}{3}\right)^2 +\left(\frac{2}{3}\right)^2 + \left(-\frac{2}{3}\right)^2\right] +\left[\frac{1}{3} \cdot \frac{2}{3}- \frac{2}{3} \cdot \frac{2}{3} +\frac{2}{3} \cdot \frac{1}{3}\right] } = 1;$$

 \item[{(viii)}] By  item (v) and Definition~\ref{D5};
 \item[{(ix)}]

By  item (v) and Lemma \ref{MJ2}

   $$ \sqrt{\| \mathbf{x}\otimes\mathbf{x}^*\|} =
   \sqrt{\| (A+B) {{\mathbf{1}_\mathbb{T}}}\oplus (-B,B,-B)   \|}$$ $$ = \sqrt{\| (A+B) {{\mathbf{1}_\mathbb{T}}}\|} =
 \sqrt{A +B} = \|\mathbf{x}\|; $$
\item[{(x)}]

If $\mathbf{x} = \mathbf{y}= \Lambda$, then the triangle inequality holds trivially.

If one of coordinates of $\mathbf{x}$ and $\mathbf{y}$ is non-zero, then apply the item (v) such that without loss of generality  we may suppose  that both $\mathbf{x}$, $\mathbf{y}$ are of the following form:  $\mathbf{x}= (0,X_\mathbf{u}, X_\mathbf{v})$, $\mathbf{y}= (0,Y_\mathbf{u}, Y_\mathbf{v})$, respectively.
Hence, also $\mathbf{x} + \mathbf{y} = (0, X_\mathbf{u}+Y_\mathbf{u},  X_\mathbf{v}+Y_\mathbf{v})$.

By Definition~\ref{D5} and Lemma \ref{L8},
 $$\|\mathbf{x}\|= \sqrt{ X_\mathbf{u}^2 +X_\mathbf{v}^2 +  X_\mathbf{u}X_\mathbf{v}} = \sqrt{ \frac{X_\mathbf{u}}{2}^2 +\frac{X_\mathbf{v}}{2}^2 +  \frac{ \left(X_\mathbf{u}+X_\mathbf{v}\right)^2}{2} }.$$

Analogously,
$$\|\mathbf{x}\|=  \sqrt{ \frac{Y_\mathbf{u}}{2}^2 +\frac{Y_\mathbf{v}}{2}^2 +  \frac{ \left(Y_\mathbf{u}+Y_\mathbf{v}\right)^2}{2} }$$
and

$$\|\mathbf{x+y}\|=  \sqrt{ \frac{(X_\mathbf{u}+Y_\mathbf{u})}{2}^2 +\frac{(X_\mathbf{v}+Y_\mathbf{v})}{2}^2 +  \frac{ \left(X_\mathbf{u}+Y_\mathbf{u}+X_\mathbf{v}+Y_\mathbf{v}\right)^2}{2} }.$$\\

For vectors $(X_\mathbf{u}, X_\mathbf{v}, X_\mathbf{u}+ X_\mathbf{v}), (Y_\mathbf{u}, Y_\mathbf{v}, Y_\mathbf{u}+ Y_\mathbf{v})$ and their sum \\ $(X_\mathbf{u}+ Y_\mathbf{u} , X_\mathbf{v} +Y_\mathbf{v} , X_\mathbf{u}+ X_\mathbf{v}+ Y_\mathbf{u}+ Y_\mathbf{v} )$, the triangle inequality in Euclidean space gives
\begin{multline*}  \sqrt{ (X_\mathbf{u}+Y_\mathbf{u})^2 +(X_\mathbf{v}+Y_\mathbf{v})^2 +  (X_\mathbf{u}+Y_\mathbf{u}+X_\mathbf{v}+Y_\mathbf{v})^2 }  \leqslant\\
\sqrt{ X_\mathbf{u}^2 +X_\mathbf{v}^2 +  (X_\mathbf{u}+X_\mathbf{v})^2 } + \sqrt{ Y_\mathbf{u}^2 +Y_\mathbf{v}^2 +  (Y_\mathbf{u}+Y_\mathbf{v})^2 }.\end{multline*}
Dividing this inequality by $\sqrt{2}$, it is easy to see that
$$ \|\mathbf{x} +\mathbf{y} \| \leq \|\mathbf{x}\| + \| \mathbf{y}\|; $$

 \item[{(xi)}] A numeric verification by Definition~\ref{D5}.
 \qed
 \end{enumerate}
\section {\bf Possible applications }
An advances generalized   of complex numbers $\mathbb{C}$ to $\mathbb{R}^n, n \geq 3$, with applications to mathematical physics can be found in \cite{Lipatov}. The authors used  group-theoretical methods, more exactly cyclic groups $C_n$ to  "complexify"~$\mathbb{R}^n$.

Intended  applications of mathematical result of the present paper are limited to three dimensions and modelling  of human colour vision. Other direction of our possible applications is   mathematical modelling of pipe organ sound (in connection to real spacial acoustics, mensuration of  pipes, generalized Pythagorean tuning which depends also on musical timbre), cf. grant  VEGA~2/0106/19 (The wooden pipe fund of historical organ positives on Slovakia, 2019--2022).

\footnotesize


\begin{thebibliography}{99}

\bibitem{Alpay} Alpay, A. -- Vajiac, A. -- Vajiac, M. B.: Gleason problem associated to a real ternary algebra and Applications, Adv. Appl. Clifford Algebras (2018) 28:43, pp. 16 (published online).


\bibitem{DK94} Drozd,  Yu. A. --  Kirichenko, V. V.:\textit{Finite Dimensional Algebras}, Springer-Verlag, Berlin, 1994.

\bibitem{Gregor-Haluska1}  Gregor, T. --
 Halu\v{s}ka, J.: Lexicographical ordering and field operations in the complex plane. Stud. Mat. 41(2014), 123--133.

\bibitem{Haluska2}  Halu\v{s}ka, J.: On fields inspired with a polar HSV-RGB theory of colour. In: Hejduk, J.-- Kowalczyk, S. -- Pawlak, R. -- Turowska, M. M.: \textit{Modern Real Analysis}, Wydawnictwo Universytetu Lodzkiego, Lodz 2015, pp.69--88. ISBN~978-83-7969-663~-5.

\bibitem{Harkin}  Harkin, A. A. --  Harkin, J. B.: Geometry of general complex numbers. Mathematics magazine, 77(2004), 118--129.

\bibitem{Lipatov}  Lipatov L. N. -- Rausch de Trautenberg M. -- Volkov G. G.: On the ternary complex analysis and its application, Journal of Mathematic physics 49(2008), 013502.

\bibitem{Nechaev} Nechaev, V. I.: {\it Number systems} (in Russian), Prosveshchenie, Moscow, 1975.

\bibitem{Rudin} Rudin, W.: \textit{Functional Analysis}, second ed., Mc Graw-Hill, Hamburg 1973, 1991, pp.~424.

\bibitem{Pierce} R. S. Pierce, \textit{Associative algebras}, Springer-Verlag, New York-Berlin, 1982.
\bibitem{Snyder}  Snyder, H. H.: An introduction to theories of regular functions on linear associative mathematics,  75--94. In: Draper R.N. (ed. ):  \textit{Commutative algebra. Analytic Methods.} Lecture Notes in Pure and Applied Mathematics 68(1982).

\bibitem{Werner} Werner, S.: {\it Topological Fields}, Mathematical Studies 157, Notas Mathematica (126), Ed.: Leopoldo Nachbin, Nord-Holland -- Amsterdam -- N.Y.-- Oxford -- Tokyo 1989. pp. 563.


\bibitem{arxiv} arxiv: 1901.08448
\end{thebibliography}
\end{document}